\documentclass[amscd,amssymb,11pt]{amsart}

\usepackage[all]{xy}

\newtheorem{thm}{Theorem}[section]
\newtheorem{lem}[thm]{Lemma}
\newtheorem{cor}[thm]{Corollary}
\newtheorem{prop}[thm]{Proposition}
\newtheorem{conj}[thm]{Conjecture}

\theoremstyle{definition}

\newtheorem{defn}[thm]{Definition}
\newtheorem{defns}[thm]{Definitions}

\newtheorem{notation}[thm]{Notation}

\newtheorem{exs}[thm]{Examples}

\theoremstyle{remark}

\newtheorem{rem}[thm]{Remark}
\newtheorem{rems}[thm]{Remarks}

\numberwithin{equation}{section}

\newcommand{\thmref}[1]{Theorem~\ref{#1}}
\newcommand{\corref}[1]{Corollary~\ref{#1}}
\newcommand{\secref}[1]{\S\ref{#1}}

\newcommand{\propref}[1]{Proposition~\ref{#1}}
\newcommand{\lemref}[1]{Lemma~\ref{#1}}

\newcommand{\hocolim}{\operatorname*{hocolim}}

\newcommand{\Map}{\operatorname{Map}}

\newcommand{\A}{{\mathcal  A}}

\newcommand{\C}{{\mathcal  C}}

\newcommand{\U}{{\mathcal  U}}

\newcommand{\Z}{{\mathbb  Z}}

\newcommand{\R}{{\mathbb  R}}

\newcommand{\Sinfty}{\Sigma^{\infty}}
\newcommand{\Oinfty}{\Omega^{\infty}}

\newcommand{\sm}{\wedge}

\newcommand{\ra}{\rightarrow}
\newcommand{\xra}{\xrightarrow}

\begin{document}

\title[Nonrealization results via Goodwillie towers]{Topological nonrealization results via the Goodwillie tower approach to iterated loopspace homology}

\author[Kuhn]{Nicholas J.~Kuhn}
\address{Department of Mathematics \\ University of Virginia \\ Charlottesville, VA 22904}
\email{njk4x@virginia.edu}
\thanks{This research was partially supported by grants from the National Science Foundation}

\date{June 19, 2008.}


\begin{abstract}  We prove a strengthened version of a theorem of Lionel Schwartz that says that certain modules over the Steenrod algebra cannot be the mod 2 cohomology of a space.  What is most interesting is our method, which replaces his iterated use of the Eilenberg--Moore spectral sequence by a single use of the spectral sequence converging to $H^*(\Omega^nX;\Z/2)$ obtained from the Goodwillie tower for $\Sinfty \Omega^n X$.  Much of the paper develops basic properties of this spectral sequence.

\end{abstract}

\maketitle

\section{Introduction and main results} \label{introduction}

In this article, I prove some constraints on the mod 2 cohomology of spaces in which Steenrod squares `jump over gaps'.  Said otherwise, for certain unstable $\A$--modules $M$ with operations jumping over gaps, there are no spaces $X$ having $\tilde H^*(X;\Z/2) \simeq M$.  Here $\A$ is the mod 2 Steenrod algebra, and a module $M$ is unstable if $Sq^sx = 0$ for all $x \in M$ and $s > |x|$.

In \cite{schwartz}, Lionel Schwartz established an interesting theorem of this type. The structure of his proof went as follows. Given $M \in \U$ of a specified sort, one wishes to show that no space $X$ exists with $H^*(X;\Z/2) \simeq M$.  Assuming the existence of such an $X$, he showed that there could be no unstable {\em algebra} structure compatible with the $\A$--module structure on $H^*(\Omega^n X;\Z/2)$, where $n$ is a number determined by $M$.  Here we recall that an unstable algebra satisfies both the Cartan formula, $Sq^k(x\cup y) = \sum_{i+j=k} Sq^ix \cup Sq^jy$, and the Restriction axiom, $Sq^{|x|}x = x^2$.

The essence of his argument is elegant, and makes clever use of the product structure in the Eilenberg--Moore Spectral Sequence for computing $H^*(\Omega X;\Z/2)$, in conjunction with the structural form of the Adem relation for $Sq^{2^k}Sq^{2^k}$.  Less elegant is his $n$--fold iterated use of the EMSS, necessitating inductive bookkeeping arguments.

The main point of our paper here is to give a new proof of Lionel's theorem, keeping the `fun' parts of his proof, but just using a single spectral sequence: the one associated to the Goodwillie tower for the functor sending a space $X$ to the spectrum $\Sinfty \Omega^n X$.  Our proof ultimately
\begin{itemize}
\item yields a strengthened version of Lionel's theorem,
\item gives some geometric meaning to what is being studied (the 2nd stage of the tower), and
\item illustrates the efficacy of using Goodwillie towers to study classical questions.
\end{itemize}

To state our main theorem we need to describe some unstable $\A$--modules.  Inside $H^*(B\Z/2;\Z/2) = \Z/2[t]$, the $\A$--module $\A \cdot t$ has basis $\{t,t^2,t^4, \dots \}$, with $Sq^{2^k}t^{2^k} = t^{2^{k+1}}$.  For $k < l$, let $\Phi(k,l) \in \U$ be the subquotient with basis $\{t^{2^k}, \dots, t^{2^l} \}$.

The modules we will be concerned with have the form $M \otimes \Phi(k,k+2)$, where $M$ is an unstable $\A$--module concentrated in degrees between $c$ and $e$. If $2^k > e-c$, then this unstable module is three copies of $M$, with $Sq^{2^k}$ sending the first copy isomorphically to the second, and $Sq^{2^{k+1}}$ sending the second copy to the third.

In formulae, the $\A$--module structure is described as follows: given $x \in M$ and $0\leq s \leq e-c$,
$ Sq^{s}(x \otimes t^{2^i}) = Sq^s(x) \otimes t^{2^i}$,
$ Sq^{s+2^k}(x \otimes t^{2^k}) = Sq^s(x) \otimes t^{2^{k+1}}$, and
$ Sq^{s+2^{k+1}}(x \otimes t^{2^{k+1}}) = Sq^s(x) \otimes t^{2^{k+2}}$.

In pictures, $M \otimes \Phi(k,k+2)$ looks like
\begin{equation} \label{M tensor Phi picture}
\xymatrix{
 *+[F]{{\hspace{.1in}}M_0{\hspace{.1in}}} \ar@/^1pc/[rr]^{Sq^{2^k}}_{\sim} && *+[F]{{\hspace{.1in}}M_1{\hspace{.1in}}} \ar@/^1pc/[rrrr]^{Sq^{2^{k+1}}}_{\sim} &&&& *+[F]{{\hspace{.1in}}M_2{\hspace{.1in}}}  }
\end{equation}

\noindent where $M_i$ is the $2^{k+i}$th suspension of $M$.

Our main theorem goes as follows.

\begin{thm} \label{main theorem} Let $M$ be an unstable $\A$--module concentrated in degrees between $c$ and $e$, such that its desuspension $\Sigma^{-1}M$ is {\em not} unstable.  Suppose there exists a space $X$ such that $\tilde H^*(X;\Z/2) \simeq \Sigma^n M \otimes \Phi(k,k+2)$.  \\

\noindent{\bf (a)} \ If $n=0$, then $2^k \leq e-c$ must hold. \\

\noindent{\bf (b)} \ If $n>0$, then $2^k \leq 4e-2c + 2n$ must hold.  If, in addition, cup products are trivial in $\tilde H^*(X;\Z/2)$, then $2^k \leq 4e-2c + 2n-2$ must hold.
\end{thm}

By constrast, Schwartz' theorem \cite[Thm.0.2]{schwartz} just says that, for all $n$, $2^k \leq 12(e+n)$ must hold.  If, in addition, cup products are trivial in $\tilde H^*(X;\Z/2)$, then $2^k \leq 12(e+n-1)$ must hold.

We note that the first statement in part (b) of the theorem (and also in Schwartz' theorem as just presented) is implied immediately by the second statement: if $X$ realizes $\Sigma^n M \otimes \Phi(k,k+2)$, then $\Sigma X$ realizes $\Sigma^{n+1} M \otimes \Phi(k,k+2)$, and cup products will be trivial in $\tilde H^*(\Sigma X;\Z/2)$.

\begin{exs} \label{examples} The theorem appears to be reasonably delicate.

Let $M = \Z/2$, so that $c=e=0$.

When $n=0$, part (a) tells us the obvious fact that $\Phi(k,k+2)$ can't be realized for all $k\geq 0$.

When $n=1$, part (b) tells us that $\tilde H^*(X;\Z/2) \simeq \Sigma \Phi(k,k+2)$ only if $k=0$ and the cohomology ring satisfies Poincar\'e duality in dimension 5.  This does in fact happen, when $X = SU(3)/SO(3)$.

When $n=2$, part (b) tells us that $\tilde H^*(X;\Z/2) \simeq \Sigma^2 \Phi(k,k+2)$ only if $k\leq 1$.  In \secref{example section}, we will look a bit more carefully at the proof of part (b) in this case, and we will show that $\tilde H^*(X;\Z/2) \simeq \Sigma^2 \Phi(1,3)$ only if the cohomology ring satisfies Poincar\'e duality in dimension 10.  This does in fact happen: a direct construction of a space with this cohomology was outlined by B.Gray on the AlgTop Discussion List\footnote{See his message of June 6, 2008, at https://lists.lehigh.edu/pipermail/algtop-l/.}.
\end{exs}

\begin{rems} {\bf (a)}  The most famous result of `mind the gap' type is due to J.F.Adams \cite{adams}, and applies to spectra as well as spaces: if $k\geq 4$, $x\in H^d(X;\Z/2)$, and $H^{d+i}(X;\Z/2) = 0$ for $0\leq i \leq 2^k - 2^{k-2}$, then $Sq^{2^k}x$ must be in the image of $Sq^1$.

{\bf (b)}  My interest in such questions goes back to my 1994 study \cite{k1} of spaces $X$ having $H^*(X;\Z/2)$ finitely generated over the mod 2 Steenrod algebra $\A$.  Using Adams' theorem, I proved that, under the extra hypothesis that $Sq^1$ acted trivially in high degrees, $H^*(X;\Z/2)$ would then have to be a finite dimensional $\Z/2$--vector space.  Furthermore, without the extra $Sq^1$ hypothesis, the conjecture that this would still be true was reduced to various questions about the nonrealizability of various sorts of unstable $\A$--modules having operations jumping over gaps. In response to my paper, Lionel formulated and proved his theorem, as it suffices to prove my conjecture \cite[\S 1]{schwartz}: see Appendix \ref{appendix} for a short discussion about how this goes.

{\bf (c)}  A much stronger qualitative theorem is conjecturally true.  The following is a restatement of the Local Realization Conjecture of \cite{k1}.

\begin{conj}  Let $M$ be an $\A$--module concentrated in a finite number of degrees.  Then for $k>>0$, there is no space or spectrum $X$ with
$$ H^*(X;\Z/2) \simeq M \otimes \Phi(k,k+1).$$
\end{conj}
In pictures, $M \otimes \Phi(k,k+1)$ looks like
\begin{equation*}
\xymatrix{
 *+[F]{{\hspace{.1in}}M_0{\hspace{.1in}}} \ar@/^1pc/[rr]^{Sq^{2^k}}_{\sim} && *+[F]{{\hspace{.1in}}M_1{\hspace{.1in}}}  }
\end{equation*}

\noindent where $M_i$ is the $2^{k+i}$th suspension of $M$.

{\bf (d)} Statement (a) of the theorem admits a simple straightforward proof that avoids all spectral sequences.  Our proofs of both parts will make clear that many other modules are ruled out for topological realization besides those explicitly appearing in the theorem. (The same comment could be made about Schwartz's paper.)  There is also a hint, in our discussion of realizing $\Sigma^2\Phi(1,3)$ in \secref{example section}, that more systematic use of Nishida relations might rule out new classes of modules.

{\bf (e)} Schwartz' theorem holds for all primes.  Thus far, we have only worked out the details with mod 2 coefficients, but our work here can certainly be modified for odd primes.  It similarly seems likely that our methods here will lead to streamlined proofs of the various other related nonrealization theorems that Schwartz and his students have proved  \cite{schwartz2, dehon-gaudens}.  By using the single Goodwillie tower spectral sequence in our argument here, we have been able to make more delicate use of the unstable module structure of $M$ than does Schwartz, and the author expects that subtle questions about how the Nilpotent and Krull filtrations of $\U$ are reflected as one passes from $H^*(X;\Z/p)$ to $H^*(\Omega^n X;\Z/p)$ can be best approached using our techniques.
\end{rems}

\begin{notation} We use the following notation.  $H^*(X)$ will mean $H^*(X;\Z/2)$. If $x \in M$ is an element of a graded vector space, then $\sigma x$ is the corresponding element of the suspended vector space $\Sigma M$.  If $X$ is a space, we will write $\Sigma^{-n}X$ for the desuspended suspension spectrum $\Sigma^{-n}\Sinfty X$.  As in \cite{may}, $\C(n,j)$ denotes the space of $j$ little $n$-cubes in a big $n$--cube.  This has a free action by the $j$th symmetric group $\Sigma_j$, and, for $X$ a space or spectrum, we let $D_{n,j}X = (\C(n,j)_+ \sm X^{\sm j})_{h \Sigma_j}$.  Note that $D_{1,j}X \simeq X^{\sm j}$ and $D_{\infty,j}X \simeq X^{\sm j}_{h \Sigma_j}$. By convention, $D_{n,0}X = S^0$ for all $n$ and $X$.
\end{notation}

The rest of the paper is organized as follows.  For much of it - sections \secref{tower section}, \secref{cohomology of extended powers}, and \secref{d1 section}, supported by Appendix \ref{op proofs} - we describe some of the general properties of the spectral sequence for computing $H^*(\Omega^n X)$.  Assuming this material, the proof of \thmref{main theorem} is satisfyingly short, and given in \secref{main theorem proof}.  Illustrating the methods of our proof, in \secref{example section} we look more carefully at how things go when $\tilde H^*(X) \simeq \Sigma^2 \Phi(1,3)$.

A version of our argument here has been known by the author for nearly a decade; indeed, I gave a talk on `A simple proof of Schwartz' non--realization theorem' at the Midwest Topology conference of October, 23, 1999.  I apologize for the delay in writing this up, and plead that this project led me to become infatuated with Goodwillie towers\footnote{To be honest, the needed geometric details underpinning the spectral sequence used here were only worked out later in joint work with S.Ahearn \cite{ak}}.  I am happy to be reunited with an earlier love: the category $\U$.

The author would like to thank Mark Mahowald and Brayton Gray for aid with Examples \ref{examples}.

\section{The Arone--Goodwillie tower of $\Sinfty \Omega^n X$} \label{tower section}

For $n<\infty$, one has a functor sending a based space $X$ to the suspension spectrum $\Sinfty \Omega^n X$. For $n=\infty$, one similarly has a functor sending a spectrum $X$ to the spectrum $\Sinfty \Oinfty X$. In either case, T.~Goodwillie's general theory of the calculus of functors \cite{goodwillie1,goodwillie2,goodwillie3} yields natural towers of fibrations
\begin{equation*}
\xymatrix{&&& \vdots \ar[d] \\
&&& P^n_3(X) \ar[d] \\
&&& P^n_2(X) \ar[d] \\
\Sinfty \Omega^nX \ar[rrr]^-{\epsilon_1}  \ar[urrr]^-{\epsilon_2} \ar[uurrr]^-{\epsilon_3} &&& P^n_1(X),}
\end{equation*}
such that the connectivity of the maps $\epsilon_j$
increases linearly with $j$ as long as $X$ is $n$--connected if $n<\infty$, and is $0$--connected if $n = \infty$.

Using G.~Arone's explicit model for this tower \cite{arone}, properties of these towers were explored in \cite{ak}.

For $n<\infty$, the spectrum $P^n_1(X)$ identifies with the spectrum $\Sigma^{-n}X$, so that $\epsilon_1$ corresponds to the evaluation map
$$ \Sinfty \Omega^n X \ra \Sigma^{-n}X,$$
and the fiber of the map $P_j^n(X) \ra P_{j-1}^n(X)$ is naturally equivalent to the spectrum $ D_{n,j}\Sigma^{-n}X$. Similarly, when $n=\infty$, the $j$th fiber is equivalent to $D_{\infty, j}X$, and $\epsilon_1$ corresponds to the evaluation map
$$ \Sinfty \Oinfty X \ra X.$$

Applying $H^*$ to the towers yields 2nd quadrant spectral sequences.  From what we have said above, one can immediately conclude the following.

When $n<\infty$, the spectral sequence converges strongly to $ H^*(\Omega^nX)$ if $X$ is an $n$--connected space, and has
$$ E_1^{-j,*} = \Sigma^j H^*(D_{n,j}\Sigma^{-n}X).$$

When $n=\infty$, the spectral sequence converges strongly to $ H^*(\Omega^{\infty}X)$ if $X$ is a 0--connected spectrum, and has
$$ E_1^{-j,*} = \Sigma^j H^*(D_{\infty,j}X).$$

For all $n$, $E_{\infty}^{*,*}$ is the graded object associated to the  filtration of $ H^*(\Omega^nX)$,
$$ \dots \supseteq F^{-3}H^*(\Omega^nX) \supseteq F^{-2}H^*(\Omega^nX) \supseteq F^{-1}H^*(\Omega^nX), $$
where $F^{-j}H^*(\Omega^nX) = Im \{\epsilon_j^*: H^*(P_j^n(X)) \ra H^*(\Omega^n X)\}$.

The spectral sequences are compatible as $n$--varies.  More precisely, the natural evaluation maps
\begin{equation} \label{evaluation} \Sigma^r \Sinfty \Omega^{n+r}X \ra \Sinfty \Omega^n X,
\end{equation}
as well as the natural equivalences (with $X_n = \Oinfty \Sigma^n X$)
\begin{equation} \Sinfty \Omega^n X_n \xra{\sim} \Sinfty \Oinfty X
\end{equation}
induce maps of towers, and then spectral sequences.

\begin{rem} When $n=1$, one recovers the classical Eilenberg--Moore spectral sequence with $ E_1^{-j,*} = H^*(X)^{\otimes j}$.  For general $n< \infty$, general Goodwillie calculus considerations imply that
the spectral sequence constructed here must necessarily agree with the dual of the spectral sequence studied by V.~Smirnov in \cite[Chapter 6]{smirnov}.
\end{rem}

\section{The mod 2 cohomology of $D_{n,*}X$} \label{cohomology of extended powers}

To use our spectral sequence, we need to have a useful description of the bigraded object
$H^*(D_{n,*}X)$.  In this section, we give this, by constructing various natural operations.  It is more traditional to describe $H_*(D_{n,*}X)$ using Dyer--Lashof operations, Browder operations, and the Pontryagin product \cite{clm, bmms}, and our operations are easily verified to be appropriately `dual' to these: see \propref{dual prop}.  Because of this, we will be brief with some verifications of properties (many of which we will {\em not} need in the proof of \thmref{main theorem}).

For simplicity, we make the blanket assumption: for all spectra $X$, $H_*(X)$ is bounded below and is of finite type.

\subsection{Structure maps}

\begin{defns} \ {\bf (a)} Let $\epsilon: \Sigma^r D_{n+r,j}X \ra D_{n,j} \Sigma^r X$ denote the map induced by the evaluation map (\ref{evaluation}). (See \cite{ak} for an explicit formula.)

{\bf (b)} \ Let $\mu: D_{n,i}X \sm D_{n,j}X \ra D_{n,i+j}X$ denote the map induced by the inclusion $\Sigma_i \times \Sigma_j \subset \Sigma_{i+j}$.

{\bf (c)} \  Let $t: D_{n,i+j}X \ra D_{n,i}X \sm D_{n,j}X$ denote the composite of the maps
$$ (\C(n,i+j)_+\sm X^{\sm{i+j}})_{h\Sigma_{i+j}} \ra (\C(n,i+j)_+\sm X^{\sm{i+j}})_{h\Sigma_{i} \times \Sigma_j}$$
and
$$ (\C(n,i+j)_+\sm X^{\sm{i+j}})_{h\Sigma_{i} \times \Sigma_j} \ra (\C(n,i)_+ \sm \C(n,j)_+\sm X^{\sm{i+j}})_{h\Sigma_{i} \times \Sigma_j},$$
where the first map is the transfer associated to $\Sigma_{i} \times \Sigma_j \subset \Sigma_{i+j}$ and the second map is induced by the $\Sigma_{i} \times \Sigma_j$--equivariant inclusion of spaces
$$ \C(n,i+j) \subset \C(n,i) \times \C(n,j).$$

{\bf (d)} \  Let $w: D_{n,2j}X \ra D_{\infty,2}D_{n,j}X$ denote the composite of the maps
$$ (\C(n,2j)_+ \sm X^{2j})_{h \Sigma_{2j}} \ra (\C(n,2j)_+ \sm X^{2j})_{h \Sigma_2 \wr \Sigma_{j}}$$
and
$$(\C(n,2j)_+ \sm X^{2j})_{h \Sigma_2 \wr \Sigma_{j}} \ra (\C(n,j)^2_+ \sm X^{2j})_{h \Sigma_2 \wr \Sigma_{j}},$$
where the first map is the transfer associated to the inclusion $\Sigma_{2} \wr \Sigma_j \subset \Sigma_{2j}$ and the second map is induced by the $\Sigma_{2} \wr \Sigma_j$--equivariant inclusion of spaces
$$ \C(n,2j) \subset \C(n,j)^2.$$
\end{defns}

\subsection{Operations}

\begin{defn}  For $r \geq 0$, define natural operations
$$\hat Q_r: H^d(D_{n,j}X) \ra H^{2d+r}(D_{n,2j}X)$$
as follows.

 Given $x \in H^d(X)$, viewed as a map $x: X \ra \Sigma^d H\Z/2$, we let $\hat Q_0(x) \in H^{2d}(D_{\infty,2}X)$ be the composite
$$ D_{\infty,2}X \xra{D_{\infty,2}x} D_{\infty,2}\Sigma^dH\Z/2 \xra{u} \Sigma^{2d}H\Z/2,$$
where $u$ represents the bottom class in $H^*(D_{\infty,2}\Sigma^dH\Z/2)$.

 Given $x \in H^d(D_{n,j}X)$, we let $\hat Q_0(x) \in H^{2d}(D_{n,2j}X)$ be the composite
$$ D_{n,2j}X \xra{w} D_{\infty,2}D_{n,j}X \xra{\hat Q_0(x)} \Sigma^{2d}H\Z/2,$$

and then, for $r >0$, we let $\hat Q_r(x) \in H^{2d+r}(D_{n,2j}X)$ be the composite
$$ D_{n,2j}X \xra{\epsilon} \Sigma^{-r}D_{n-r,2j}\Sigma^r X \xra{ \Sigma^{-r}\hat Q_0(\Sigma^rx)} \Sigma^{2d+r}H\Z/2.$$
\end{defn}

\begin{defn}  Define a natural product
$$*:H^*(D_{n,i}X) \otimes H^*(D_{n,j}X) \ra H^{*}(D_{n,i+j}X)$$
to be the map on cohomology induced by the `transfer' maps
$$t: D_{n,i+j}X \ra D_{n,i}X \sm D_{n,j}X.$$
\end{defn}

Note that, when $n=1$, the $*$--product is the standard shuffle product on the tensor algebra $TH^*(X)$.

\begin{defn}  Define a natural coproduct
$$\Psi:H^{*}(D_{n,i+j}X) \ra H^*(D_{n,i}X) \otimes H^*(D_{n,j}X)$$
to be the map on cohomology induced by the maps
$$\mu: D_{n,i}X \sm D_{n,j}X \ra D_{n,i+j}X.$$
\end{defn}
\begin{defn} For $n < \infty$ and $d_1+\dots+d_j = d$, define
$$ L_{n-1}: H^{d_1}(X) \otimes \dots \otimes H^{d_j}(X) \ra H^{d+(j-1)(n-1)}(D_{n,j}X)$$
to be the map on cohomology induced by the map
$$\epsilon: D_{n,j}X \ra \Sigma^{1-n}D_{1,j}\Sigma^{n-1}X = \Sigma^{(j-1)(n-1)}X^{\sm j}.$$
\end{defn}
Note that $L_0$ is just the usual product in the tensor algebra $TH^*(X)$.

The following will be made precise in Appendix \ref{op proofs}. See \propref{dual prop}.

\begin{prop} \label{dumb dual prop} In a suitable sense, the cohomology operations $\hat Q_r$, $*$, and $L_{n-1}$ are dual to the homology operations $Q_r$, $*$, and $\lambda_{n-1}$.
\end{prop}

\subsection{Some properties of the operations}

\begin{prop} \label{Hopf alg prop} The $*$--product and $\Psi$--coproduct makes $H^*(D_{n,*}X)$ into a bigraded bicommutative Hopf algebra.
\end{prop}

\begin{prop} \label{unstable prop} For all $x \in H^*(D_{n,j}X)$, $\hat Q_r(x) = 0$ for $r \geq n$.
\end{prop}

\begin{prop} \label{epsilon prop} Under $\epsilon^*: H^*(D_{n,*}\Sigma X) \ra H^{*-1}(D_{n+1,*}X)$, the operations behave as follows. \\

\noindent (i) \  $\epsilon^*$ sends $*$--decomposables to 0: $\epsilon^*(x*y) = 0$ for all $x \in H^*(D_{n,i}\Sigma X)$ and $y \in H^*(D_{n,j}\Sigma X)$, with $i \geq 1$ and $j \geq 1$.  Similarly, the image of $\epsilon^*$ is contained in the $\Psi$--primitives. \\

\noindent (ii) \ $\epsilon^*$ commutes with the $\hat Q$ operations: $\epsilon^*(\hat Q_r(\sigma x)) = \hat Q_{r+1}(x)$. \\

\noindent (iii) \ $\epsilon^*$ commutes with the $L$ operations: for all $x_1, \dots, x_k \in H^*(X)$, $$\epsilon^*(L_{n-1}(\sigma x_1 \otimes \dots \otimes \sigma x_k)) = L_{n}(x_1 \otimes \dots \otimes x_k).$$
\end{prop}
\begin{proof} Parts (ii) and (iii) are clear from the definition, and part (i) is only slightly less so.  For more detail about (i), see \cite[Example 6.7]{ak}.
\end{proof}

\begin{prop}  For all $x \in H^*(X)^{\otimes k}$, $\hat Q_{n-1}(L_{n-1}(x)) = L_{n-1}(x \otimes x)$.
\end{prop}
\begin{proof} By parts (ii) and (iii) of the last proposition, this reduces to the case when $n=1$, where it reads \  $\hat Q_0(x) = x\otimes x$, for $x \in H^*(X^{\sm k})$, and this is clear from the definition of $\hat Q_0$.
\end{proof}

\begin{prop} \label{Q additive prop} For all $x,y \in H^d(D_{n,j}X)$, the following identities hold.\\

\noindent{(i)} \ $\hat Q_0(x+y) = \hat Q_0(x) + \hat Q_0(y) + x*y$. \\

\noindent{(ii)} \ $\hat Q_r(x+y) = \hat Q_r(x) + \hat Q_r(y)$, for all $r>0$. \\

\noindent{(iii)} \ $x*x=0$.
\end{prop}

See Appendix \ref{op proofs} for a proof.

\begin{prop} \label{L relations prop} For all $n \geq 2$, the kernel of $L_{n-1}: TH^*(X) \ra H^*(D_{n,*}X)$ is the span of the shuffle product decomposables.
\end{prop}
\begin{proof} This is dual to the well known statement that the image of $\epsilon_*: H_*(D_{n,*}X) \ra TH_*(X)$ is the vector space of primitives, which identifies as the free restricted Lie algebra generated by $H_*(X)$.  Note that \propref{epsilon prop}(i) implies that the kernel is at least as big as claimed.
\end{proof}

One has Adem relations among the $\hat Q_r$.

\begin{prop} \label{Q adem relations prop}
$$ \hat Q_r \hat Q_s(x) = \sum_j \binom{j-r}{2j-r-s} \hat Q_{r+2s-2j} \hat Q_{j}(x).$$
\end{prop}

This follows from the homology Adem relations, using \propref{dual prop}.  Similarly, the calculation of $H_*(D_{n,*}X)$ as in \cite{clm, bmms} implies the next theorem.

\begin{thm} Using the $*$--product, $H^*(D_{n,*}X)$ is the graded commutative algebra generated by the elements of the form $\hat Q_{r_1}\dots \hat Q_{r_l}L_{n-1}(x_1 \otimes \dots \otimes x_k)$, subject to the relations listed in \propref{unstable prop}, \propref{Q additive prop}, \propref{L relations prop}, and \propref{Q adem relations prop}.
\end{thm}

Finally, we have Nishida relations. Compare with {\cite[p.40]{milgram}}, {\cite[Thm. II.3.5]{bmms}}, and {\cite[Prop. 6.12]{ksw}}.

\begin{prop}  \label{Nishida prop} For all $x \in H^d(D_{n,j}X)$, the following identities hold.\\

\noindent{(i)} \ $\displaystyle Sq^s\hat Q_0(x) = \sum_t \binom{d-t}{s-2t} \hat Q_{s-2t}(Sq^tx) +  \sum_{t<s/2} Sq^{t}x*Sq^{s-t}x$. \\

\noindent{(ii)} \ $\displaystyle Sq^s\hat Q_r(x) = \sum_t \binom{d+r-t}{s-2t}\hat Q_{r+s-2t}(Sq^tx)$.
\end{prop}
See Appendix \ref{op proofs} for more about this.

\section{Some properties of the spectral sequence for $H^*(\Omega^n X)$} \label{d1 section}

Here we collect some basic properties of the spectral sequences arising from the towers of \secref{tower section}. From \cite{ak}, we learn the following.
\begin{prop} The spectral sequence is a spectral sequence of differential graded Hopf algebras, with the product and coproduct structure on $E_1$ given by the $*$ and $\Psi$, converging to the usual Hopf algebra structure on $H^*(\Omega^nX)$.
\end{prop}

From the geometric construction of the spectral sequence, we deduce the next proposition.
\begin{prop} The spectral sequence is a spectral sequence of $\A$--modules, with $\A$ acting columnwise on $E_1$ in the evident way, and converging to the usual $\A$--module structure on $H^*(\Omega^n X)$. In particular, $F^{-j}H^*(\Omega^n)$ is a sub $\A$--module of  $H^*(\Omega^n X)$ for all $j$.
\end{prop}

Finally we determine the differential $d_1$ from the -2 line to the -1 line.  In other words, we determine the homomorphism
$$ d_1: \Sigma^2H^*(D_{n,2}\Sigma^{-n}X) \ra \Sigma H^*(\Sigma^{-n}X),$$
induced by the connecting map $\delta$ in the cofibration sequence
$$ D_{n,2}\Sigma^{-n}X \ra P_2^n(X) \ra \Sigma^{-n}X \xra{\delta} \Sigma D_{n,2}\Sigma^{-n}X.$$

\begin{prop} \label{d1 prop} For $x, y \in H^*(X)$ we have the following formulae. \\

\noindent{(i)} \ $d_1(\sigma^2L_{n-1}(\sigma^{-n}x \otimes \sigma^{-n}y)) = \sigma^{1-n}(x \cup y)$. \\

\noindent{(ii)} \ $d_1(\sigma^2\hat Q_r(\sigma^{-n} x)) = \sigma^{1-n} Sq^{r+ |x|-n + 1}(x)$. \\

\noindent{(iii)} \ $d_1(x*y) = 0$.
\end{prop}

\begin{proof}  Formula (iii) is clear, as $d_1$ is a derivation.

Formula (i) reduces to the case when $n=1$, where it becomes the well known formula
$$ d_1(x \otimes y) = x \cup y$$
in the bar construction associated to the Eilenberg--Moore spectral sequence.

Formula (ii) reduces to the case when $r=0$, and then to the case when $n=\infty$, where we wish to show that, for $X$ a 0-connected spectrum and $x \in H^d(X)$,
$$d_1(\sigma^2\hat Q_0(x)) = \sigma Sq^{d+1}x.$$
As the left side of the equation is natural, there must be an element $a \in \A^{d+1}$ such that
$$d_1(\sigma^2 \hat Q_0(x)) = \sigma ax.$$
To show that $a$ must be $Sq^{d+1}$, we note that when $X = \Sigma^dH\Z/2$, so that the spectral sequence converges to $H^*(K(\Z/2,d))$, $\Sigma^{-1} E_{\infty}^{-1,*}$ will be $F(d)$, the free unstable quotient of $\Sigma^d\A$. For connectivity reasons, the only way for this to happen is if $d_1(\sigma^2\hat Q_0(u)) = \sigma Sq^{d+1}u$, when $u$ is the bottom class of $\Sigma^d \A$.
\end{proof}
\begin{rem} The proposition should be compared to the homology formulae in \cite[6.2]{smirnov}; in particular, Smirnov's formula on page 124, three lines before his second theorem.
\end{rem}
\begin{cor} \label{d1 corollary} In the spectral sequence computing $H^*(\Omega^n X)$ with $X$ an $n$--connected space, $\Sigma^{-1}E_2^{-1,*}$ will be the maximal unstable quotient of $$\Sigma^{-n}(H^*(X)/(\text{$\cup$-decomposables})).$$
Even more is true if $\tilde H^*(X) \simeq \Sigma^n M$ with $M \in \U$, and has no nontrivial cup products: then $E_3^{-1,*} = E_2^{-1,*} = E_1^{-1,*}$, and $E_2^{-2,*} = E_1^{-2,*}$.
\end{cor}
\begin{proof}  The first statement follows evidently from the previous proposition.  In the situation of the second statement, the assumption then tells us that $d_1:E_1^{-2,*} \ra E_1^{-1,*}$ is identically zero.  Since $E_1^{-3,*}$ is spanned by $*$--decomposables, the fact that $d_1$ is a derivation allows us to conclude that $d_1:E_1^{-3,*} \ra E_1^{-2,*}$ is also identically zero.  Thus we have that both $E_2^{-2,*} = E_1^{-2,*}$ and $E_2^{-1,*} = E_1^{-1,*}$. It follows that $E_2^{-3,*}$ is again spanned by algebra decomposables, and so, as before, we conclude that $d_2:E_2^{-3,*} \ra E_2^{-1,*}$ is identically zero.
\end{proof}
The is a similar corollary in the $n=\infty$ case.
\begin{cor}
In the spectral sequence computing $H^*(\Omega^{\infty} X)$ with $X$ a $0$--connected spectrum, $\Sigma^{-1}E_2^{-1,*}$ will be the unstable quotient of $H^*(X)$. Even more is true if $\tilde H^*(X) \simeq M$ with $M \in \U$: then $E_3^{-1,*} = E_2^{-1,*} = E_1^{-1,*}$, and $E_2^{-2,*} = E_1^{-2,*}$.
\end{cor}

\section{Proof of \thmref{main theorem}} \label{main theorem proof}

Recall the assumptions on $M$ in the theorem.  We have numbers $c\leq e$ such that $M^s \neq 0$ only if $c\leq s\leq e$.  The statement that $M$ is {\em not} the desuspension of an unstable module means precisely that there exists $x \in M$ such that $Sq^{|x|}x \neq 0$.  We fix such an element and let $d=|x|$, so that $c \leq d \leq 2d \leq e$.

Assuming that $2^k > e-c$, it is easily verified that $M \otimes \Phi(k,k+2)$ is the module as pictured in (\ref{M tensor Phi picture}):
\begin{equation*}
\xymatrix{
 *+[F]{{\hspace{.1in}}M_0{\hspace{.1in}}} \ar@/^1pc/[rr]^{Sq^{2^k}}_{\sim} && *+[F]{{\hspace{.1in}}M_1{\hspace{.1in}}} \ar@/^1pc/[rrrr]^{Sq^{2^{k+1}}}_{\sim} &&&& *+[F]{{\hspace{.1in}}M_2{\hspace{.1in}}}  }
\end{equation*}
\noindent where $M_i = M \otimes \langle t^{2^{k+i}} \rangle$.  We let $a = x \otimes t^{2^k} \in M_0$, $b = x \otimes t^{2^{k+1}} \in M_1$, and $c = x \otimes t^{2^{k+2}} \in M_2$.  Thus $|a| = d + 2^k$, $|b| = d + 2^{k+1}$, $|c| = d + 2^{k+2}$, $Sq^{2^{k}}a = b$, $Sq^{2^{k+1}}b = c$, and $Sq^dc \neq 0$.

With this notation, we give the quick proof of \thmref{main theorem}(a).  Assuming that $2^k > e-d$, we show that there is no unstable algebra structure on $M \otimes \Phi(k,k+2)$, so that there can be no space $X$ such that $\tilde H^*(X) \simeq M \otimes \Phi(k,k+2)$.

The proof of this is simple.
\begin{equation*}
\begin{split}
Sq^{2^k}(a \cup b) & = b \cup b \text{ \ (by the Cartan formula)} \\
  & = Sq^{d+2^{k+1}}b \text{ \ (by the Restriction axiom)} \\
  & = Sq^d c \neq 0.
  \end{split}
\end{equation*}
Thus $a\cup b \neq 0$.  But $|a\cup b| = 2d+ 3\cdot 2^k$, which we claim is a degree in the `gap' between $M_1$ and $M_2$, so that $a\cup b =0$, giving us a contradiction. In other words, we claim that
$$e+2^{k+1} < 2d + 3\cdot 2^k < c + 2^{k+2}.$$
The first inequality follows by adding $2^{k+1}$ to the inequalities
$$ e < c + 2^k \leq 2d +2^k,$$
while the second inequality follows by adding $3\cdot 2^{k}$ to the inequalities
$$ 2d \leq e < c +2^k.$$

We now begin the longer proof of \thmref{main theorem}(b). So let $n>0$, and suppose that $\tilde H^*(X) \simeq \Sigma^n M \otimes \Phi(k,k+2)$, and has trivial cup products.  We can assume that $X$ is a CW complex. For technical reasons\footnote{This will be needed to ensure that $\sigma Sq^d(c) \in E_1^{-1,2d+2^k+1}$ is not in the image of $d_3$: see \lemref{nonzero lemma}.  Replacing $X$ by $Y$ would not be needed if $d \leq 2c$.}, rather than working with $X$, we work with the quotient $Y = X/X_{d+n+2^k-1}$.  Since $H^*(Y) \ra H^*(X)$ is an isomorphism for $*> d+n+2^k$ and is epic if $*= d+n+2^k$, one easily deduces that $\tilde H^*(Y) \simeq \Sigma^n N$ with $N \in \U$, still has trivial cup products, and  $N$ is as pictured:
\begin{equation} \label{1-line}
\xymatrix{
 *+[F]{{\hspace{.05in}}N_0{\hspace{.1in}}} \ar@/^1pc/@{->>}[rr]^{Sq^{2^k}} && *+[F]{{\hspace{.1in}}M_1{\hspace{.1in}}} \ar@/^1pc/[rrrr]^{Sq^{2^{k+1}}}_{\sim} &&&& *+[F]{{\hspace{.1in}}M_2{\hspace{.1in}}}  }
\end{equation}
\noindent where $N_0$, $M_1$, and $M_2$ are nonzero only in degrees in the intervals $[d+2^k,e+2^k]$, $[c+2^{k+1},e+2^{k+1}]$, and $[c+2^{k+2},e+2^{k+2}]$.

Choosing a `new' $a \in N_0$ mapping onto the `old' $a \in M_0$, we have as before: $|a| = d + 2^k$, $|b| = d + 2^{k+1}$, $|c| = d + 2^{k+2}$, $Sq^{2^{k}}a = b$, $Sq^{2^{k+1}}b = c$, and $Sq^dc \neq 0$.

We assume the inequality
\begin{equation} \label{constraint} 2^k > 4e-2c+2n-2, \end{equation} and we show that this leads to a contradiction by showing that $H^*(\Omega^nY)$, as computed by our spectral sequence, can not admit an unstable algebra structure.

As a first observation, we note that \corref{d1 corollary} applies, so that $E_3^{-1,*} = E_2^{-1,*} = E_1^{-1,*}$, and $E_2^{-2,*} = E_1^{-2,*}$.

A picture of $\Sigma^{-1}E_1^{-1,*}$ is given by (\ref{1-line}), and is all permanent cycles. Thus there exist $\alpha \in H^{d+2^k}(\Omega^n Y)$, $\beta \in H^{d+2^{k+1}}(\Omega^n Y)$, and $\gamma \in H^{d+2^{k+2}}(\Omega^n Y)$, respectively represented by $a$, $b$, and $c$, and we have $Sq^{2^k}\alpha = \beta$, and $Sq^{2^{k+1}}\beta = \gamma$.

A picture of $\Sigma^{-2}E_1^{-2,*}$ in degrees less than $2c+2^{k+2}$ is given by
\begin{equation} \label{2-line}
\xymatrix{
 *+[F]{N_0\cdot N_0}  && *+[F]{N_0\cdot M_1} }
\end{equation}
where $N_0\cdot N_0$ denotes the span of all elements of the form $\hat Q_r(x)$, $x*y$, or $L_{n-1}(x \otimes y)$ with $x,y \in N_0$, and $N_0\cdot M_1$ denotes the span of $x*y$ and $L_{n-1}(x \otimes y)$ with $x \in N_0$ and $y \in M_1$.

$N_0\cdot N_0$ is nonzero only in the interval $[2d+2^{k+1},2e+2^{k+1} + n-1]$, and includes the element $\hat Q_0(a)$ in lowest degree. As $E_2^{-2,*} = E_1^{-2,*}$, this is a permanent cycle\footnote{The key point is that, since $Sq^{d+1}x=0$, $d_1 \hat Q_0(a) = Sq^{d+2^k+1}a = Sq^{d+1}b = 0$ also.}, and so represents an element $\delta \in H^{2d+2^{k+1}}(\Omega^n Y)$.

$N_0\cdot M_1$ is nonzero only in the interval $[c+d+ 3\cdot 2^k, 2e + 3\cdot 2^k + n-1]$, and includes the element $a*b$, which represents $\alpha \cup \beta \in H^{2d+3\cdot 2^k}(\Omega^n Y)$.

\begin{lem} $Sq^{2^k}\delta = \alpha \cup \beta$.
\end{lem}
\begin{proof} Our constraint (\ref{constraint}) implies that $2^k-2d \geq n$ and also that $2^{k-1}>e$. Using these inequalities, one easily checks that the formula for $Sq^{2^k}\hat Q_0(a)$ given by \propref{Nishida prop} simplifies to yield $$Sq^{2^k}\hat Q_0(a) = a*Sq^{2^k}a = a*b.$$  As both $Sq^{2^k}\delta$ and $\alpha \cup \beta$ are represented by $a*b$, it follows that $Sq^{2^k}\delta - \alpha \cup \beta$ is represented by something in bidegree $(-1,2d+3\cdot 2^k +1)$.  But there is nothing nonzero in this bidegree because (\ref{constraint}) implies that $2^k>e-c$, and this then implies that $e+2^{k+1}<2d+3\cdot 2^k< c+2^{k+2}$.
\end{proof}

\begin{lem} $Sq^{2^k}(\alpha \cup \beta) = Sq^d \gamma$
\end{lem}
\begin{proof} $Sq^{2^k}(\alpha \cup \beta) = \beta^2 = Sq^{d+2^{k+1}}\beta = Sq^d \gamma$.
\end{proof}

\begin{lem} \label{nonzero lemma} $Sq^d\gamma \neq 0$.
\end{lem}
\begin{proof} The lowest degree differential with potentially nonzero image in the $-1$--line would be
$$d_3: E_3^{-4,4d+2^{k+2}+4} \ra E_3^{-1,4d+2^{k+2}+2}.$$
Thus $Sq^dc \in E_3^{-1,2d+2^{k+2}+1}$ is not a boundary.
\end{proof}

\begin{cor} $Sq^{2^k}Sq^{2^k}\delta \neq 0$.
\end{cor}

We will now use the next lemma to show that $Sq^{2^k}Sq^{2^k}\delta =0$ if our numerical constraint (\ref{constraint}) holds, and this contradiction will finish the proof
of \thmref{main theorem}(b).

Let $\A(k)$ be the subalgebra of $\A$ generated by $Sq^1, \dots, Sq^{2^k}$.

\begin{lem} \cite[Lemma 2.6]{schwartz} $Sq^{2^k}Sq^{2^k} \in \A(k-1)Sq^{2^k}\A(k-1)$.
\end{lem}

\begin{cor} $Sq^{2^k}Sq^{2^k}\delta = 0$.
\end{cor}
\begin{proof}  This is a calculation taking place in $F^{-2}H^*(\Omega^n Y)$, which in the relevant degrees looks like
\begin{equation*}
\xymatrix{
 *+[F]{M_1 + N_0\cdot N_0}  &&& *+[F]{N_0\cdot M_1}  &&& *+[F]{M_2 + M_1\cdot M_1}  }
\end{equation*}
The element $\delta \in M_1+N_0\cdot N_0$, while $Sq^{2^k}Sq^{2^k}\delta \in M_2 + M_1\cdot M_1$.  By the lemma, if both of the gaps pictured span greater than $2^{k-1}$ degrees, then the corollary would follow.

The span of the first gap equals
\begin{multline*}
(\text{the bottom degree of } N_0\cdot M_1) - (\text{the top degree of }M_1 + N_0\cdot N_0) \\
=(d+c+3\cdot 2^k) - (2e+2^{k+1}+n-1) = 2^k + d+c-2e-n+1.
\end{multline*}

The span of the second gap equals
\begin{multline*} (\text{the bottom degree of } M_2 + M_1\cdot M_1) - (\text{the top degree of }N_0\cdot M_1) \\
=(c+2^{k+2}) - (2e+3\cdot 2^k +n-1) = 2^k + c-2e-n+1.
\end{multline*}

Thus both gaps have spans bigger than $2^{k-1}$ if
$$ 2^k + c-2e-n+1 > 2^{k-1},$$
so that
$$ 2^{k-1}> 2e-c +n-1,$$
which is our constraint (\ref{constraint}).
\end{proof}

\section{Realizing $\Sigma^2 \Phi(1,3)$}\label{example section}

Suppose that $\tilde H^*(X) \simeq \Sigma^2 \Phi(1,3)$, so there exist nonzero elements $a \in H^4(X)$, $b \in H^6(X)$, and $c \in H^{10}(X)$ such that $Sq^2a=b$ and $Sq^4b=c$.  Using the spectral sequence converging to $H^*(\Omega^2 X)$ as in the last section, we prove the following.

\begin{prop} In this case, $a\cup b = c$ must hold.
\end{prop}

\begin{proof} Repressing some suspensions from the notation, Figure 1 shows all of $E_1^{*,*}$ in total degree less than or equal to 8, in the spectral sequence converging to $H^*(\Omega^2 X)$.

\begin{figure}
\begin{equation*}
\begin{array}{ccccc|c}
\hat Q_0(\hat Q_0(a))   &&&&& 12\\
& b*\hat Q_0(a)     &&&& 11\\
&a*\hat Q_1(a) &\hat Q_0(b)&&& 10\\
&a*\hat Q_0(a)  &L_1(a\otimes b)&c&& 9\\
& &a*b&&& 8\\
& &\hat Q_1(a)&&& 7\\
 & &\hat Q_0(a)&&& 6\\
&   &&b&& 5\\
 &   &&&& 4\\
 &    &&a&& 3\\
&     &&&& 2\\
& &&&& 1\\
& &&&  1   & 0  \\ \hline
-4& -3&-2&-1& 0&s\backslash t \\
\end{array}
\end{equation*}
\caption{$E_1^{s,t}(\Omega^2X)$ when $\tilde H^*(X) \simeq \Sigma^2\Phi(1,3)$} \label{figure 1}
\end{figure}

The only possible differential here is $d_1(L_1(a\otimes b)) = c$ which, by \propref{d1 prop}, happens exactly when $a\cup b = c \in H^*(X)$.  Assuming this does {\em not} happen, through degree 8, $F^{-2}H^*(\Omega^2X)$ would have a basis given by elements 1, $\alpha$, $\beta$, $\delta$, $\epsilon$, $\alpha \cup \beta$, $\lambda$, $\gamma$, and $\omega$, in respective degrees 0, 2, 4, 4, 5, 6, 7, 8, and 8, and represented by $1$, $a$, $b$, $\hat Q_0(a)$, $\hat Q_1(a)$, $a * b$, $L_1(a\otimes b)$, $c$, and $\hat Q_0(b)$.  The structure of $\Phi(1,3)$ shows that $\gamma = \beta^2 = \alpha^4$.  Furthermore, the arguments in the last section show that $Sq^2Sq^2\delta = \gamma \neq 0$.

The relation $Sq^2Sq^2 = Sq^1 Sq^2 Sq^1$ then implies that $Sq^1\delta \neq 0$.  However, the Nishida formula, \propref{Nishida prop}, implies that $Sq^1\hat Q_0(a) = 0$, and thus $Sq^1\delta = 0$. This contradiction implies that $d_1(L_1(a\otimes b)) = c$ must have been true, so that $a \cup b = c \in H^{10}(X)$, $\lambda \in H^7(\Omega^2X)$ doesn't exist, and $\gamma = 0 \in H^8(\Omega^2X)$.

\end{proof}
\appendix

\section{More proofs of the properties of the operations} \label{op proofs}

\begin{proof}[Proof of \propref{Q additive prop}]  Thanks to \propref{epsilon prop}(i), formula (ii) follows from the formula (i).  Letting $x=y$ in (i) implies (iii).

To prove (i), given $x,y \in H^d(D_{n,j}X)$, $\hat Q_0(x+y)$ is represented by the composite
$$ D_{n,2j}X \xra{w} D_{\infty,2}D_{n,j}X \xra{D_{\infty,2}(x+y)} D_{\infty,2}\Sigma^d H\Z/2 \xra{u} \Sigma^{2d}H\Z/2.$$

It is standard \cite[Cor.II.1.6]{bmms} that, given $x,y: Y \ra Z$, $D_{\infty,2}(x+y)$ decomposes as the sum of $D_{\infty,2}(x)$, $D_{\infty,2}(y)$, and the composite
$$ D_{\infty,2}Y \xra{t}  Y^{\sm 2} \xra{x\sm y} Z^{\sm 2} \xra{\mu}  D_{\infty,2}Z.$$
It follows that $\hat Q_0(x+y)= \hat Q_0(x) + \hat Q_0(y)$ plus the composite
\begin{equation*}
\xymatrix{
D_{n,2j}X \ar[d]^{w} && & \Sigma^{2d} H\Z/2   \\
D_{\infty,2}D_{n,j}X \ar[r]^-t & (D_{n,j}X)^{\sm 2} \ar[r]^-{x\sm y} & (\Sigma^d H\Z/2)^{\sm 2} \ar[r]^-{\mu} & D_{\infty,2}\Sigma^d H\Z/2. \ar[u]_u}
\end{equation*}
But this last map is just $x*y$, as there is a commutative diagram
\begin{equation*}
\xymatrix{
D_{n,2j}X \ar[d]^{w} \ar[dr]^-t && & \Sigma^{2d} H\Z/2   \\
D_{\infty,2}D_{n,j}X \ar[r]^-t & (D_{n,j}X)^{\sm 2} \ar[r]^-{x\sm y} & (\Sigma^d H\Z/2)^{\sm 2} \ar[r]^-{\mu} \ar[ur]^-u & D_{\infty,2}\Sigma^d H\Z/2. \ar[u]_u}
\end{equation*}
Here the left triangle commutes due to the transitivity of the transfer with respect to the inclusions $(\Sigma_j)^2 \subset \Sigma_2 \wr \Sigma_j \subset \Sigma_{2j}$.
\end{proof}

We now make precise the `duality' proposition \propref{dumb dual prop}.  In the following proposition, given $y,z \in H_*(X)$, $Q_r(y),y*z,\lambda_{n-1}(y,z) \in H_*(D_{n,2}X)$ denote the usual elements under the Dyer-Lashof operation $Q_r$, the Pontryagin product $*$, and the Browder operation $\lambda_{n-1}$ of \cite[Part III]{clm}.
\begin{prop} \label{dual prop}  Let $\langle x,y \rangle$ denote the cohomology/homology pairing.  For $n >1$, given $w,x \in H^*(X)$ and $y,z \in H_*(X)$, the following formulae hold.

{\bf (a)} \ $
\langle {\hat Q_r x},{Q_s y} \rangle =
\begin{cases}
\langle {x},{y} \rangle & \text{if } r = s \\ 0 & otherwise.
\end{cases}
$

{\bf (b)} \ $
\langle {\hat Q_r x},{y*z} \rangle =
\begin{cases}
\langle {x},{y} \rangle \langle {x},{z} \rangle& \text{if } r = 0 \\ 0 & otherwise.
\end{cases}
$

{\bf (c)} \ $
\langle {\hat Q_r x},{\lambda_{n-1}(y,z)} \rangle =0.
$

{\bf (d)} \ $
\langle {w*x},{Q_sy} \rangle =
\begin{cases}
\langle {w},{y} \rangle \langle {x},{y} \rangle & \text{if } s = 0 \\ 0 & otherwise.
\end{cases}
$

{\bf (e)} \ $
\langle {w*x},{y*z} \rangle =
\langle {w},{y} \rangle \langle {x},{z} \rangle + \langle {w},{z} \rangle \langle {x},{y} \rangle.
$

{\bf (f)} \ $
\langle {w*x},{\lambda_{n-1}(y,z)} \rangle = 0.
$

{\bf (g)} \ $
\langle {L_{n-1}(w \otimes x)},{Q_sy} \rangle = 0.
$

{\bf (h)} \ $
\langle {L_{n-1}(w \otimes x)},{y*z} \rangle = 0.
$

{\bf (i)} \ $
\langle {L_{n-1}(w \otimes x)},{\lambda_{n-1}(y,z)} \rangle =
\langle {w},{y} \rangle \langle {x},{z} \rangle + \langle {w},{z} \rangle \langle {x},{y} \rangle.
$
\begin{proof}[Sketch proof] The behavior of the homology operations under the evaluation $\epsilon: \Sigma^s D_{n+s,2}X \ra D_{n,2} \Sigma^s X$ well known \cite[Thm.III.1.4]{clm}: $\epsilon_*(\sigma^s Q_{n+s}y) = Q_s(\sigma^s y)$, $\epsilon_*(\sigma^s y*z)=0$, and $\epsilon_*(\sigma^s\lambda_{n+s-1}(y,z)) = \lambda_{n-1}(\sigma^sy,\sigma^sz)$. Note in particular, that, under $\epsilon: D_{n,2}X \ra \Sigma^{n-1} X \sm X$, one has $\epsilon_*(\lambda_{n-1}(y,z)) = \sigma^{n-1} y\otimes z + z \otimes y$.

Similarly, the behavior under $t: D_{n,2}X \ra X \sm X$ is easy to describe: $t_*(Q_{s}y) = 0$, $t_*(y*z)=y\otimes z + z\otimes y$, and $t_*(\lambda_{n-1}(y,z)) = 0$.

Using this information, the various formulae are easily verified, using the naturality of the cohomology/homology pairing.
\end{proof}
\end{prop}

\begin{proof}[Proof of \propref{Nishida prop}]  \propref{epsilon prop}(i) again implies that the formula when $r>0$ follows from the formula when $r=0$.  Furthermore, by the construction of the operations, we can assume that $n=\infty$ and $j=1$, and so one just needs to verify (i) for $Sq^s \hat Q_0(x) \in H^{2d+s}(D_{\infty,2}X)$.

This can be proved in various ways.  One is to use the previous proposition together with the usual Nishida relations.

Another approach goes as follows. One verifies (i) for various sorts of spectra $X$.

If $X$ is a suspension spectrum, then the cohomology of $D_{\infty,2}X$ is detected by the two maps $ X \sm X \ra D_{\infty,2}X$ and $B\Z/2_+ \sm X \ra D_{\infty,2}X$, and one checks that the elements on both sides of formula (i) map to the same elements under these detection maps.

If $X = S^{-c}$, then $D_{\infty,2}X = \Sigma^{-c}\R P^{\infty}_{-c}$, and one can directly check the formula, working within the $\A$--module $\Z/2[t,t^{-1}]$.

If (i) is true for $x \in H^*(X)$ and $y \in H^*(Y)$, then it is true for $x \otimes y \in H^*(X \sm Y)$.  To see this, one uses the map
$$ D_{\infty,2}(X\sm Y) \ra D_{\infty,2}X\ \sm D_{\infty,2}Y$$
which sends $\hat Q_0(x) \otimes \hat Q_0(y)$ to $\hat Q_0(x \otimes y)$.

If (i) is true for spectra $X_c$ and $\displaystyle X = \hocolim_c X_c$ then (i) is true for $X$.  This follows since then $\displaystyle H^*(D_{\infty,2}X) = \lim_c H^*(D_{\infty,2}X_c)$ (using our standing finite type hypothesis).

Assembling all these special cases yields the formula for general spectra $X$, as $\displaystyle X \simeq \hocolim_c S^{-c} \sm \Sinfty X_c$, where $X_c$ is the $c^{th}$-space of $X$.
\end{proof}

\section{The nonrealization conjecture of \cite{k1}} \label{appendix}

Following \cite{k1, schwartz}, we review how \thmref{main theorem} implies

\begin{thm}  If $H^*(Z)$ is a finitely generated $\A$--module, then it is a finite dimensional $\Z/2$--vector space.
\end{thm}

\begin{proof}[Sketch proof]Let $\bar T: \U \ra \U$ be the reduced Lannes functor, left adjoint tensoring with $\tilde H^*(B\Z/2)$. Let $\Delta: \text{Spaces} \ra \text{Spaces}$ be defined by $ \Delta(Z) = \Map(B\Z/2, Z)/Z$,
where $Z$ embeds in $\Map(B\Z/2, Z)$ as the space of constant maps.  Under good circumstances, $\bar TH^*(Z) \simeq H^*(\Delta(Z))$.

Suppose that $H^*(Z)=L$ is infinite, but finitely generated over $\A$.  Replacing $Z$ and $L$ by their suspensions if needed, we can assume that `good circumstances' will hold.  As $L$ is a finitely generated $\A$--module, $\bar T^iL$ will again be finitely generated for all $i$, and $\bar T^lL = 0$ for some $l$. Since $L$ is also infinite, the smallest such $l$ will be at least 2.  Choosing this smallest $l$, let $Y = \Delta^{l-2}(Z)$.  Then $N=H^*(Y)=\bar T^{l-2}H^*(Z)$ will still be infinite and finitely generated over $\A$, but now also $\bar T^2 N = 0$. (These reductions are made in \cite{k1}.)

Now we use a structure theorem: $N \in \U$ is finitely generated over $\A$ and satisfies $\bar T^2N=0$ if and only if it fits into an exact sequence in $\U$ of the form
$$ 0 \ra A \ra N \ra M \otimes \Phi(j,\infty) \ra B \ra 0,$$
for some finite dimensional unstable modules $A$, $B$, and $M$, and for some $j$, where $\Phi(j,\infty) = \A\cdot t^{2^j} \subset \Z/2[t]$. Furthermore $M=\bar TN$. (A weaker version of this appears in \cite{k1}, with the full statement appearing in \cite{schwartz}.)

It now easily follows that, given any large enough $k$, an appropriate `subquotient' $X$ of $Y$ will satisfy $H^*(X) = M \otimes \Phi(k,k+2)$.  This contradicts \thmref{main theorem}.
\end{proof}

\end{document}